\documentclass[12 pt]{article}
\usepackage{graphicx,mathtools,amsmath,amssymb,amsthm,enumerate, hyperref, xcolor,mathrsfs}
\usepackage{marvosym, fontawesome}
\usepackage[indent = 20 pt]{parskip}

\usepackage{url}
\usepackage{enumitem}

\newtheorem{theorem}{Theorem}[section]
\newtheorem{corollary}[theorem]{Corollary}
\newtheorem{lemma}[theorem]{Lemma}
\newtheorem{proposition}[theorem]{Proposition}
\newtheorem{question}[theorem]{Question}

\newtheorem{letterthm}{Theorem}

\theoremstyle{definition}
\newtheorem{definition}[theorem]{Definition}

\newtheorem{example}[theorem]{Example}

\theoremstyle{remark}
\newtheorem{remark}[theorem]{Remark}

\usepackage{etoolbox}
\preto\theorem{\medskip}

\newcommand{\cR}{\mathcal{R}}

\newcommand{\cU}{\mathcal{U}}
\newcommand{\cZ}{\mathcal{Z}}

\newcommand{\N}{\mathbb{N}}
\newcommand{\F}{\mathbb{F}}

\newcommand{\T}{\mathbb{T}}
\newcommand{\Z}{\mathbb{Z}}

\DeclareMathOperator{\Sub}{Sub}

\newcommand{\actson}{\curvearrowright}

\DeclareMathOperator{\id}{id}

\DeclareMathOperator{\Aut}{Aut}

\DeclareMathOperator{\Stab}{Stab}
\DeclareMathOperator{\Fix}{Fix}

\DeclareMathOperator{\lcm}{lcm}

\usepackage[doi=false, url =false , isbn = false,style=alphabetic,sorting=nyt, backend = biber, maxcitenames=50, maxalphanames = 5, maxnames=50]{biblatex}
\makeatletter
\def\thanks#1{\protected@xdef\@thanks{\@thanks
        \protect\footnotetext{#1}}}
\makeatother

\begin{document}

\title{On essential freeness of actions of mixed identity-free groups}
\author{
Soham Chakraborty \thanks{
\hspace{-2 em} \faMapMarker\hspace{0.18 em}: École Normale Supérieure, Paris, France \\ 
\Letter : \texttt{soham.chakraborty@ens.psl.eu}
}
} 

\date{\today}

\setlength{\parindent}{0em}

\maketitle 

\begin{abstract}\noindent
It was asked in \cite{alekseev_etal} if every faithful, discrete, ergodic pmp action of a mixed identity-free (MIF) group is essentially free. In this short note we give some positive evidence by showing that for a MIF group, any faithful generalized Bernoulli action and any affine action on a compact abelian group  is essentially free. This gives some classes of discrete ergodic actions with this property beyond the totally ergodic or compact case.  
\end{abstract}

\section{Introduction}

Let \(F_n=\langle x_1,\ldots,x_n\rangle\) be the free group on $n$ generators. A countable discrete group \(G\) satisfies a \textit{mixed identity} if there is some \(n\geq 1\) and some nontrivial word \(w\in G*F_n\) such that $w(g_1,\ldots,g_n)=e$ for all $g_1,\ldots,g_n\in G$. The group \(G\) is called \textit{mixed identity-free} (MIF) if it satisfies no nontrivial mixed identities. By \cite[Remark 5.1]{Hull_Osin}, $G$ is MIF if and only if there are no mixed-identities in $G \ast \Z$. MIF groups have strong structural properties, for example, they cannot be direct products and they always have the infinite conjugacy class (ICC) property. Natural examples of MIF groups include every acylindrically hyperbolic group with trivial finite radical (in particular every free group) by \cite{Hull_Osin}, and Zariski dense subgroups of certain Lie groups by \cite{Tomanov}. 

An ergodic probability measure preserving (pmp) action of a group $G$ on a standard probability space $(X,\mu)$ is called \textit{discrete} if its image in the full group of the orbit equivalence relation $[\cR(G \actson X)]$ is discrete with the topology of convergence in measure. Motivated by the study of discrete images of MIF groups in full groups of  their orbit equivalence relations, the authors in \cite{alekseev_etal} ask the following:
\medskip
\begin{question}(c.f. \cite[Question 1.4]{alekseev_etal})
\label{Q: main question}
    Is every faithful, ergodic, discrete pmp action of a MIF group essentially free? 
\end{question}
The authors in \cite{alekseev_etal} attribute this question to Yair Glasner. It is especially interesting for the free group $\F_2$ as it is not known if there is a discrete image of $\F_2$ in the full group of the unique ergodic hyperfinite II$_1$ equivalence relation. Question \ref{Q: main question} has a positive answer for compact actions by \cite[Theorem 1.3]{alekseev_etal}. In general, even without the MIF condition, if $G$ has trivial finite radical, then any mixing (more generally, totally ergodic i.e., every subgroup acts ergodically) action of $G$ is essentially free by \cite{TuckerDrob}. We give some new examples of actions (which are neither totally ergodic, nor compact) for which there is a positive answer to Question \ref{Q: main question}. In Theorems \ref{Theorem: main theorem} and \ref{Thm: affine actions} we prove the following:
\medskip
\begin{letterthm}
    Let $G$ be a MIF group and let $G \actson (X,\mu)$ either be a faithful ergodic generalized Bernoulli action or a faithful affine action on a compact abelian group with its normalized Haar measure. Then $G \actson (X,\mu)$ is essentially free. 
\end{letterthm}

In general, both of these classes of action are not totally ergodic (hence not mixing) and not compact. Generalized Bernoulli actions are weakly mixing, and hence it is interesting to ask if Question \ref{Q: main question} at least has a positive answer for weakly mixing actions. The MIF assumption is also necessary for both classes, as we demonstrate in Remarks \ref{Rem: MIF necessary for gen bernoulli} and \ref{Rem: MIF necessary for algebraic action}.

Recall that an invariant random subgroup (IRS) on $G$ is a conjugacy invariant measure on the space of subgroups $\Sub(G)$ with the Borel structure induced from its usual Chabauty topology. Every pmp action $G \actson (X,\mu)$ comes with an associated stabilizer IRS that is denoted by $(\Stab)_* \mu$. In fact every IRS arises in this way as shown in \cite{Kesten_IRS}. An IRS $\nu$ is called \textit{ergodic} if it is the stabilizer IRS of an ergodic action. An action is faithful if and only if the \textit{normal core} $N_\nu$ of the IRS is trivial (see Definition \ref{Def: normal core}). We show the following in Proposition \ref{Prop: IRS main proposition} Corollary \ref{Corr: MIF groups with atomic IRS}: 

\begin{letterthm}
    Let $\nu$ be an ergodic atomic IRS on a MIF group $G$. If $N_\nu = \{e\}$, then $\nu = \delta_{\{e\}}$. In particular, if every ergodic IRS of $G$ is atomic, then every faithful ergodic action of $G$ is essentially free. 
\end{letterthm}

Once again, such a result is not true for groups that are not MIF, as we demonstrate in Remark \ref{Rem: MIF necessary  for IRS}. A major class of groups that have only atomic ergodic IRSs are character rigid groups, but in this case it directly follows that faithful ergodic actions are essentially free and hence the above proposition does not add anything. Nevertheless we can apply it to some other MIF groups like some torsion free Tarski Monsters (see example \ref{Ex: tarski monsters}).
\medskip

\textbf{Acknowledgements:} The author is supported by the ERC advanced grant
101141693 titled \textit{Noncommutative ergodic theory of higher rank lattices}. The author would like to thank François Le Maître for introducing him to the central questions in \cite{alekseev_etal} and Felipe Flores for some useful comments. 

\section{Generalized Bernoulli actions}

Let $I$ be an infinite set and suppose that an infinite group $G$ acts on $I$. Suppose $(X_0,\mu_0)$ is a standard probability space (which will be assumed non-trivial throughout) and let $(X,\mu) = (X_0,\mu_0)^I$ and let $G$ act on $(X,\mu)$ by Bernoulli shifts. Such actions are called generalized Bernoulli actions of $G$. The following are observed in \cite{PopaVaes}. 
\medskip
\begin{proposition}
\label{Prop: ergodicity and freeness of actions}
    Let $G \actson (X_0,\mu)^I = (X,\mu)$ be a generalized Bernoulli action. Then the following are true: 
    \begin{enumerate}
        \item $G \actson (X,\mu)$ is faithful if and only if $G \actson I$ is faithful. 
        \item If $(X_0,\mu_0)$ is diffuse, then $G \actson (X,\mu)$ is essentially free if and only if every $g \neq e$ moves at least one element of $I$.
        \item If $(X_0, \mu_0)$ has an atom, then $G \actson (X,\mu)$ is essentially free if and only if every $g \neq e$ moves infinitely many elements of $I$. 
        \item $G \actson (X,\mu)$ is ergodic if and only if $G \actson (X,\mu)$ is weakly mixing if and only if every orbit of $G \actson I$ is infinite. 
    \end{enumerate}
\end{proposition}

\medskip

\begin{lemma}
    A faithful generalized Bernoulli action $G \actson (X,\mu) = (X_0,\mu_0)^I$ is discrete. If $(X_0,\mu_0)$ is diffuse then $G \actson X$ is in fact essentially free. 
\end{lemma}

\begin{proof}
    Let $\ell(g) = 1 - \mu(\Fix(g)) = \mu(\{x \in X \; | \; gx \neq x\})$. We need to show that $\inf_{g \neq e} \ell(g) > 0$. Let $D \subset X_0 \times X_0$ be the diagonal copy of $X_0$ and let $\delta = (\mu_0 \times \mu_0)(D)$. Clearly $\delta < 1$, and in fact when $(X_0,\mu_0)$ is diffuse, $\delta = 0$. Now let $g \neq e$ be arbitrary and pick an element $i \in I$ such that $gi \neq i$ by faithfulness. Notice that if $x \in X$ is fixed by $g$, then it follows that $x_i = x_{gi}$. So we have: 
    \begin{align*}
        \Fix(g) \subset \{x \in X \; | \; x_i = x_{gi}\} 
    \end{align*}
    Now $\mu(\{x \in X \; | \; x_i = x_{gi}\})$ is exactly equal to $\delta$ and thus we get that $\mu(\Fix(g)) \leq \delta$. Hence $\ell(g) \geq 1-\delta$ for all $g \neq e$ and thus $\inf_{g \neq e} \ell(g) >0$. When $(X_0,\mu_0)$ is diffuse, $\delta = 0$ and thus $\ell(g) = 1$ for all $g \neq e$, thus giving essential freeness. 
\end{proof}

\medskip
So the interesting case to consider is when the base $(X_0,\mu_0)$ is atomic. 

\begin{theorem}
\label{Theorem: main theorem}
    Let $G$ be a MIF group and let $G \actson (X,\mu) = (X_0,\mu_0)^I$ be a faithful ergodic generalized Bernoulli action on a base space that has an atom. Then $G \actson (X,\mu)$ is essentially free.  
\end{theorem}
\begin{proof}
    By Proposition \ref{Prop: ergodicity and freeness of actions}, we have that $G \actson I$ acts faithfully and every orbit is infinite. Letting $S(g) = \{i \in I \; | \; gi \neq i\}$, we have to show that $S(g)$ is infinite  for all $g \neq e$. Now let $e \neq g \in G$ be such that $|S(g)| = n < \infty$. We claim that for any finite subset $F \subset I$, there is an element $h \in G$ such that $h\cdot S(g) \cap F = \emptyset$. 

    To prove the claim, notice that for $s \in S(g)$ and $f \in F$, if the set $G_{s,f} = \{h \in G \; | \; h \cdot s = f\}$ is non-empty, then it is a left coset of the stabilizer $G_s$ at the point $s$. Since the orbit $G \cdot s$ is infinite by assumption, this implies that $G_s$ is an infinite index subgroup of $G$. By Neumann's coset lemma, if a group can be written as a union of cosets of subgroups, then at least one of the subgroups must have finite index. Hence: 
    \begin{align*}
       \bigcup_{s \in S(g)} \bigcup_{f \in F} G_{s,f} \neq G
    \end{align*}
    Thus there exists an element $h \in G$ such that $h \cdot S(g)$ is disjoint from $F$, as required, thus proving the claim. 

    Applying the claim repeatedly, we get group elements $h_1,h_2,...,h_{n+1}$ such that $\{ h_j \cdot S(g) \}_{1 \leq j \leq n +1}$ are pairwise disjoint. Notice that for an element $x \in G$, one has that $S(xgx^{-1}) = x\cdot S(g)$ as: 
    \begin{align*}
        i \in S(xgx^{-1}) &\iff xgx^{-1} \cdot i \neq i \iff gx^{-1} \cdot i \neq x^{-1}i \iff x^{-1}i \in S(g). 
    \end{align*}
    Now letting $k_j = h_jgh_j^{-1}$, we have that $\{S(k_j)\}_{1 \leq j \leq n+1}$ are all pairwise disjoint. Now for an arbitrary $x \in G$, $|S(xgx^{-1})| = |x \cdot S(g)| = n$ and therefore it cannot intersect these $n+1$ pairwise disjoint subsets. So for each $x \in G$, there exists an index $j = j(x)$ such that $S(xgx^{-1}) \cap S(k_j) = \emptyset$. Hence the elements $xgx^{-1}$ and $k_{j}$ have disjoint supports in $I$. Since permutations with disjoint supports commute, this implies that the commutator $[xgx^{-1}, k_{j}]$ acts trivially on $I$. By faithfulness we get that $[xgx^{-1},k_{j}] = e$. Let $X_1$ be a variable and define the word $w_j(X) = [X_1gX_1^{-1}, k_j]$ in $G \ast \langle X_1 \rangle$. Notice that each $w_j$ is nontrivial in the free product and for each $x$, at least one $w_j(x) = e$.

    Now we define new independent variables $X_2,X_3,...,X_{n+1}$ and recursively define the following words: 
    \begin{align*}
        W_1 = w_1 \text{ and } W_{j+1} = [W_j,X_{j+1}w_{j+1}X_{j+1}^{-1}]
    \end{align*}
    Once again each $W_{n+1}$ is nontrivial in the free product $G \ast \langle  X_1,X_2,...,X_{n+1} \rangle$. But for each $x$, one of $w_j(x)$'s are trivial. Correspondingly, all the subsequent $W_j$'s are trivial and hence $W_{n+1}$ is a mixed identity in $G \ast \langle X_1,X_2,...,X_{n+1} \rangle$. The word \(W_{n+1}\) is nontrivial by the normal form theorem for free products, since at the \(j\)-th step the new variable \(X_{j+1}\) occurs only in the conjugate \(X_{j+1}w_{j+1}X_{j+1}^{-1}\). Since $G$ is MIF, this gives a contradiction and $S(g)$ is infinite, as required.  
\end{proof}

\begin{example}
    The most interesting examples of generalized Bernoulli actions arise as $G \actson G/H = I$ for an infinite index subgroup $H< G$. In this setting, Theorem \ref{Theorem: main theorem} can be stated in terms of properties of $H<G$. Indeed one can check easily that $G \actson (X,\mu)$ is faithful if and only if $H$ is \textit{core-free}, i.e., $\bigcap_{k \in G} kHk^{-1} = \emptyset$. In this case the theorem says that if $G$ is MIF, then for each $g \in G$, $g$ moves infinitely many cosets in $G/H$. 
\end{example}
\medskip
\begin{remark}
    It is easy to see that generalized Bernoulli actions are not totally ergodic in general. Indeed, if $H<G$ is an infinite index subgroup and if $H$ is infinite, then $H \actson G/H$ fixes the point $eH$. Thus $H \actson (X_0,\mu_0)^{G/H}$ is not ergodic. As a consequence, such actions are not mixing or mildly mixing. 
\end{remark}
\medskip
\begin{remark}
\label{Rem: MIF necessary for gen bernoulli}
    One can check that the conclusion might fail for non MIF groups. Let $S_\infty$ be the group of finitely supported permutations on $\N$. Consider the action $S_\infty \actson \N$ and the generalized Bernoulli action $S_\infty \actson \{0,1\}^{\N} = (X,\mu)$ with equal measure on the base. Since $S_\infty \actson \N$ is faithful and transitive, $S_\infty \actson (X,\mu)$ is faithful and ergodic. Consider the order 2 permutation $p = (1\;2)$, and notice that $px = x$ if and only if $x_1 = x_2$. Therefore $\mu(\Fix(p)) = \frac{1}{4} + \frac{1}{4}= \frac{1}{2}$ and the action is not essentially free. 
\end{remark}

\section{Affine actions on compact abelian groups}

In this section we shall consider actions of countable discrete groups on compact abelian groups by continuous affine automorphisms. Since the normalized Haar measure on a compact group is unique, such an action is automatically pmp. We shall denote the compact abelian group by $K$, the unique normalized Haar measure by $\nu$ and the acting group by $G$. Notice that by uniqueness of the Haar measure, any continuous group automorphism of $(K,\nu)$ is pmp and hence any affine automorphism of $(K,\nu)$ is pmp. 
\medskip
\begin{definition}
    Let $N$ be a locally compact abelian group. Let $\alpha \in \Aut(N)$ be a continuous automorphism and we shall denote by $N_\alpha$ the subgroup $N_\alpha = \{\alpha(n) n^{-1} \; | \; n \in N\} < N$. For an action $G \actson N$ by continuous automorphisms, we shall denote by $N_g$ the subgroup corresponding to the automorphism given by $g$. 
\end{definition}
In what follows we shall denote the Pontryagin dual of an abelian group $K$ by $\widehat{K}$ and the dual automorphism of $\alpha \in \Aut(K)$ by $\widehat{\alpha} \in \Aut(\widehat{K})$. 
\medskip
\begin{lemma}
\label{Lemma: measure of fixed points}
    Let $T$ be an affine automorphism of $K$ given by $T(k) = b \alpha(k)$ where $b \in K$ and $\alpha \in \Aut(K)$. Then $\nu(\Fix(T))>0$ if and only if $\Fix(T) \neq \varnothing$ and $\widehat{K}_{\widehat{\alpha}}$ is finite.  
\end{lemma}
\begin{proof}
    By \cite[Proposition 14.D.6]{Bekka-delaharpe2020}, we get that $\Fix(\alpha)^{\perp} = \widehat{K}_{\widehat{\alpha}}$ where $\widehat{K}$ is the discrete dual and $\widehat{\alpha}$ is the dual automorphism. Moreover $\Fix(\alpha)$ is non-null if and only if $\widehat{K}_{\widehat{\alpha}}$ is finite, again by \cite[Proposition 14.D.6]{Bekka-delaharpe2020}. Now notice that if $k,l \in \Fix(T)$, then:
    \begin{align*}
        \alpha(kl^{-1}) = b\alpha(k)\alpha(l^{-1})b^{-1} = T(k)T(l)^{-1} = kl^{-1}
    \end{align*}
    So given a fixed $k \in \Fix(T)$, we get that $k \Fix(T) = \Fix(\alpha)$. By left invariance of the Haar measure, $\nu(\Fix(T)) > 0$ if and only if it is nonempty and $\nu(\Fix(\alpha)) > 0$ as required.
\end{proof}

\begin{lemma}
\label{Lemma: estimate for MIF}
    Let $N$ be a discrete abelian group. For each positive integer $m \geq 1$, there is a positive integer $q(m)$  such that whenever $\alpha, \beta \in \Aut(N)$ with $|N_{\alpha}| < m$ and $|N_\beta| < m$, then $(\gamma)^{q(m)} = \id$ for any $\gamma \in \langle \alpha,\beta \rangle$.
\end{lemma}
\begin{proof}
    Notice that $N_\alpha \cdot N_\beta$ is finite and $n \coloneqq |N_\alpha \cdot N_\beta| < m^2$. Notice that $\alpha(\alpha(n)n^{-1}) = \alpha(\alpha(n)) \alpha(n)^{-1}$ and hence $\alpha(N_\alpha) = N_\alpha$. Moreover for $n \in N_\beta$ we have: 
    \begin{align*}
         \alpha(n) = n \cdot \alpha(n) n^{-1} \in N_\alpha \cdot N_\beta
    \end{align*}
    Therefore $\alpha(N_\alpha \cdot N_\beta) \subset N_\alpha \cdot N_\beta$ and since $\alpha$ is an automorphism in fact $\alpha(N_\alpha \cdot N_\beta)= (N_\alpha \cdot N_\beta)$. By the same argument one gets $\beta(N_\alpha \cdot N_\beta) = N_\alpha \cdot N_\beta$. Consider the subgroup $\langle \alpha, \beta \rangle$ of $\Aut(N)$ and notice that for each $\gamma \in \langle \alpha, \beta \rangle$, we have $\gamma(N_\alpha \cdot N_\beta) = N_\alpha \cdot N_\beta$ and $\gamma$ acts trivially on $N/(N_\alpha \cdot N_\beta)$. Consider the following positive integer: 
    \begin{align*}
        r(n) = \lcm \{|\Aut(A)| \; | \; A \text{ finite abelian group with } |A| \leq n\}
    \end{align*}
    Then $\delta = \gamma^{r(n)}$ acts trivially on $N_\alpha N_\beta$ and on $N/(N_\alpha \cdot N_\beta)$. One checks now that $\delta^n = \id$ where $n = |N_\alpha \cdot N_\beta|$. Hence the order of $\gamma$ divides $nr(n)$. Now since $n < m^2$, we can take $q(m) = (m^2)! \lcm_{1 \leq n\leq m^2} r(n)$. Then irrespective of the choices of $\alpha$ and $\beta$, we get that any $\gamma \in \langle \alpha,\beta \rangle$ satisfies $\gamma^{q(m)} = \id$.
\end{proof}

\begin{theorem}
\label{Thm: affine actions}
    Let $G$ be a MIF group and $\rho: G \actson (K,\nu)$ be a faithful affine action on a compact abelian group with its normalized Haar measure. Then $G \actson (K,\nu)$ is essentially free. 
\end{theorem}
\begin{proof}
    Suppose that the action is given by $\rho_g(k) = b_g \alpha_g(k)$ for elements $b_g \in K$ and $\alpha \in \Aut(K)$ and suppose that $\nu(\Fix(g)) > 0$ for some nontrivial $g \in G$. Let $N = \widehat{K}$ and notice that by Lemma \ref{Lemma: measure of fixed points}, $N_g$ is finite. Let $|N_g| = m$ and notice that $|N_{hgh^{-1}}| = |N_g| = m$. By Lemma \ref{Lemma: estimate for MIF}, one can pick a positive integer $q(m)$ such that for any $h \in G$, we get $(\widehat{\alpha}_{hgh^{-1}g})^{q(m)} = (\widehat{\alpha}_{hgh^{-1}} \widehat{\alpha}_g)^{q(m)} = \id$. Thus letting $\omega(h) = hgh^{-1}g$ for all $h \in G$, we get that $(\alpha_{\omega(h)})^{q(m)} = \id$ for all $h$. Since translations commute in $K$, we have that $\rho_{[\omega(h)^{q(m)},\omega(k)^{q(m)}]} = \id$ for all $h,k \in G$. By faithfulness of the action, $[\omega(h)^{q(m)},\omega(k)^{q(m)}] = 1$ for all $h,k \in G$. Thus the word $W(x,y) = [(xgx^{-1}g)^{q(m)},(ygy^{-1}g)^{q(m)}]$ is a mixed identity in $G \ast \langle x,y \rangle$. This is a mixed identity and it is non-tirival as there are no relations between the letters $g,x$ and $y$, thus contradicting the MIF hypothesis.  
\end{proof}

\begin{remark}
    Notice that there is no ergodicity assumption required in Theorem \ref{Thm: affine actions}. In fact in the ergodic case, Theorem \ref{Thm: affine actions} does not in general fall under the framework of \cite[Theorem 1.3]{alekseev_etal}. It can be checked that an  action as above is compact as an action if and only if the orbits of the dual action are finite. 
    In the case where the action is by continuous automorphisms, ergodicity of such an action is equivalent to the non-identity orbits of the dual action being infinite (see \cite[Proposition 14.D.6]{Bekka-delaharpe2020}), which in particular contradicts compactness of the action.
\end{remark}
\medskip
\begin{remark}
\label{Rem: MIF necessary for algebraic action}
    The MIF assumption is necessary in Theorem \ref{Thm: affine actions}. For example let $G$ be the lamplighter group $\Z/2\Z \wr \Z$. Let $t$ denote the generator of $\Z$ and let $a_n$ denote the lamp at index $n$. Let $a \neq 1$ be an element in the base $B = \bigoplus_{\Z} \Z/2\Z$. Then notice that $gag^{-1} \in B$ for all $g \in G$. Thus $w(x) = [xax^{-1},a]$ is a nontrivial mixed identity in $G \ast \langle x \rangle$. Hence $G$ is not MIF. Now let $\Omega = \Z \times \Z/2\Z$ and consider the action $G \actson \Omega$ by: 
    \begin{align*}
        t (n,\epsilon) = (n+1,\epsilon) \text{ ; } a_k(n,\epsilon) = (n,\epsilon) \text{ if } n\neq k \text{ and } (n,\epsilon + 1) \text{ if } n = k
    \end{align*}
    This gives a faithful permutation action of $G$ on $\Omega$. Now consider the compact abelian group $K = (\Z/2\Z)^\Omega$ and consider the action $G \actson K$ induced by the permutation action $G \actson \Omega$. Since $G \actson \Omega$ is faithful, we get that $G \actson K$ is faithful. The dual group $\widehat{K}$ is $\bigoplus_\Omega \Z/2\Z$ and hence any element in $\widehat{K}$ corresponds to a finite subset $F \subset \Omega$. Notice that there are infinitely many pairwise distinct elements in the set $\{t^n F \; | \; n \in \Z\}$. Hence the $G$-orbit of any non-trivial element in $\widehat{K}$ is infinite. This implies that $G \actson K$ is ergodic, for example by \cite[Proposition 14.D.6]{Bekka-delaharpe2020}. 
    
    We claim that the action is discrete. Notice that for the permutation action $G \actson \Omega$, any non-trivial element of $g \in G$ either moves y many elements in $\Omega$, or is a finite product of lamps. In the first case, $\nu(\Fix(g)) = 0$, and in the second case, $\nu(\Fix(g)) \leq 2^{-i}$ for some positive integer $i$. In either case, the uniform metric $d(g,h) \geq 1/2$ and hence the action is discrete. 
    Now consider the element $a_0$ which is the lamp at the 0-coordinate. The elements $(0,1)$ and $(0,0)$ in $\Omega$ are switched by $a_0$ and every other element is fixed. Thus $\nu(\Fix(a_0)) = 1/2$ and the action is not essentially free.  
\end{remark}

\section{MIF groups with atomic ergodic IRSs}

By now there are several classes of discrete countable groups that are character rigid, and in particular have the property that every faithful ergodic pmp action is essentially free. This includes irreducible lattices in simple higher rank Lie groups \cite{Stuck_Zimmer}, projective special linear groups \cite{PetersonThom}, irreducible lattices in semisimple groups with finite center and all simple factors of rank at least 2 \cite{CreutzPeterson}, simple algebraic groups \cite{Bekka}, Higman-Thompson groups \cite{DudkoMedynets} among others. In this section, we observe that if all ergodic IRSs of a MIF group are atomic, then every faithful ergodic action is essentially free. 
\medskip
\begin{definition}
\label{Def: normal core}
    Let $\nu$ be an ergodic IRS on $G$. For each $g \in G$ let $\rho_\nu(g) = \nu(\{H \in \Sub(G) \; | \; g \in H\}$). We define the \textit{normal core} of $\nu$ by $N_\nu = \{ g \in G \; | \; \rho_\nu(g) = 1\}$
\end{definition}

Notice that when $\nu = (\Stab)_* \mu$ for an action $G \actson (X,\mu)$, then $\rho_\nu(g) = \mu(\Fix(g))$ and $N_\nu$ is the set of elements that act trivially. The main observation is the following.  
\medskip
\begin{proposition}
\label{Prop: IRS main proposition}
    Let $\nu$ be an ergodic atomic IRS on a MIF group $G$. If $N_\nu = \{e\}$, then $\nu = \delta_e$. 
\end{proposition}
\begin{proof}
    Let $H \in \Sub(G)$ be an atom for $\nu$. Since $\nu$ is an IRS, all conjugates of $H$ have measure $\nu(H)$ and hence the conjugacy class of $H$ in $\Sub(G)$ is finite. Let the conjugacy class be $E = \{H_1 = H,H_2,...,H_n\}$, then $E$ is a conjugation invariant Borel subset of $\Sub(G)$ and by ergodicity, $\nu(E) = 1$. Hence we get that $\nu = \frac{1}{n} \sum_{i= 1}^{n} \delta_{H_i}$ and $N_\nu  = \bigcap_{i=1}^n H_i$. By hypothesis, we have that $\bigcap_{i=1}^n H_i = \{e\}$, and we want to show that $H_i = \{e\}$ for all $i$. Consider the subgroup: 
    \begin{align*}
        K = \bigcap_{i=1}^n N_G(H_i) \text{ where } N_G(H_i) \text{ is the normalizer of } H_i 
    \end{align*}
    and notice that $K$ is the kernel of the conjugation action of $G$ on $E$. Therefore $K < G$ is a finite index normal subgroup. Choose a positive integer $q$ such that $g^q \in K$ for all $g \in G$. Let $L_i = H_i \cap K$ and notice that each $L_i$ is a normal subgroup of $K$. Moreover if $gH_ig^{-1} = H_j$ for some $g \in G$ we have: 
    \begin{align*}
        gL_i g^{-1} = gH_ig^{-1} \cap gKg^{-1} = H_j \cap K = L_j
    \end{align*}
    So all $L_i$'s are conjugate, and in particular, isomorphic. 

    Suppose first that $L_1$ is non-trivial. Pick non-trivial elements $l_i \in L_i$ for each $i$ and consider the word $w$ on $G \ast \langle x \rangle$ given by:
    \begin{align*}
        w(x) = [...[[[x^q,l_1],l_2],l_3]...]
    \end{align*}
    We claim that $w(x) \in \bigcap_{i=1}^n L_i$. First notice that $[x^q, l_1] \in L_1$ because $x^q \in K$ and $L_1 <K$ is a normal subgroup. Now, for an element $l \in L_1 \cap... \cap L_{j-1}$, the commutator $[l,l_j]$ is in $L_j$ again because $l_j \in L_j$ and $L_j < K$ is normal. Moreover since each $L_i$ is normal in $K$, we get $[l,l_j] \in L_1 \cap...\cap L_j$. By induction, the claim is proved. Thus the word $w(x)$ is a mixed identity as $w(g) \in \bigcap_{i=1}^n L_i \subset \bigcap_{=1}^n H_i = \{e\}$. We now claim that it is a non-trivial mixed identity. This is because a commutator (at any step) $[a,l_i]$ is trivial if and only if $a$ is in the centralizer of $l_i$, and for a free product $G \ast \langle x \rangle$, the centralizer of an element of $G$ is in $G$. At every induction step, the word $a$ has a nontrivial occurrence of the letter $x$ and hence $a \notin G$, as required. This contradicts the MIF assumption and hence each $L_i$ is trivial. 

    Now suppose that $H_1$ is non-trivial and let $h$ be a nontrivial element in $H_1$. For all $g \in G$, we have $g^q \in K < N_G(H_1)$ and hence $[g^q, h] \in H_1$.But $K < G$ is normal we also have $[g^q,h] \in K$. Thus $[g^q,h] \in K \cap H_1 = L_1 = \{e\}$.
    Thus $w(x) = [x^q,h]$ is a nontrivial mixed-identity in $G \ast \langle x \rangle$. Since $G$ is MIF, this forces $h = e$ and hence each $H_i$ is trivial, completing the proof.          
\end{proof}

Proposition \ref{Prop: IRS main proposition} immediately gives the following corollary. 
\medskip
\begin{corollary}
\label{Corr: MIF groups with atomic IRS}
    Let $G$ be a MIF group such that all its ergodic IRSs are atomic. Then every faithful ergodic pmp action of $G$ is essentially free.  
\end{corollary}
\medskip 
The following example shows that Corollary \ref{Corr: MIF groups with atomic IRS} is useful beyond the character rigidity framework.
\medskip
\begin{example}
\label{Ex: tarski monsters}
If $G$ is a MIF group with only countably many subgroups then the assumption of Corollary \ref{Corr: MIF groups with atomic IRS} is satisfied. For example if $G$ is a finitely generated simple torsion-free Tarski Monster group as constructed in \cite[Theorem E]{Coulon_etal} then they are MIF by \cite[Corollary F]{Coulon_etal} and every faithful ergodic pmp $G$-action is essentially free.
\end{example}
\medskip
\begin{remark}
\label{Rem: MIF necessary  for IRS}
    Once again we remark that the MIF assumption is crucial in Corollary \ref{Corr: MIF groups with atomic IRS}. Indeed let $G = S_3 \times \Z$ where $S_3$ is the permutation group on $E = \{1,2,3\}$. It is an easy check that we leave to the reader that $G$ has countably many subgroups, and hence every ergodic IRS of $G$ is atomic. Let $\Z \actson T$ be given by an irrational rotation $z \mapsto z + \alpha$. Consider the product action $S_3 \times \Z \actson E \times \T$ by:
    \begin{align*}
        (\sigma, n)(i,z) = (\sigma(i), z + n\alpha)
    \end{align*}
    If $(\sigma,n)$ acts trivially, then $n = 0$ and $\sigma$ is the identity permutation, so the action is faithful. If $f \in L^\infty(E \times \T)$ is $G$-invariant then each cross section $f_i(z) = f(i,z)$ is $\Z$-invariant. Since irrational rotations are ergodic, this implies that $f_i$ is essentially constant for each $i$. Invariance under $S_3$ then implies that $f$ is essentially constant, thus the action is ergodic. Now notice that if $n \neq 0$, $\Fix(\sigma,n)$ is null because $\Fix(n) \subset \T$ is null. In fact a nontrivial $\sigma$ can fix at most one point in $E$ and the measure of $\Fix(\sigma, 0)$ is less than or equal to $1/3$. Thus the action is discrete. Finally, note that for the permutation $\tau = (12)$, the set $\Fix(\tau,0)$ has measure 1/3 and the action is not essentially free. 
\end{remark}

\printbibliography
\end{document}